\documentclass[11pt,reqno]{amsart}
\usepackage{amsmath}
\usepackage{float}
\usepackage{geometry}
\usepackage{etoolbox}
\usepackage{amscd, amsfonts, amssymb, graphicx, color}
\usepackage{url}
\usepackage{tikz-cd}
\usepackage{hyphenat}
\usepackage{mathtools}
\usepackage[colorlinks=true]{hyperref}
\usepackage{titlesec}
\titleformat{\section}
{\normalfont\bfseries\large}{\thesection}{1em}{}
\makeatletter
\patchcmd\maketitle
{\uppercasenonmath\shorttitle}
{}
{}{}
\patchcmd\maketitle
{\@nx\MakeUppercase{\the\toks@}}
{\the\toks@}
{}
{}{}
\patchcmd\@settitle{\uppercasenonmath\@title}{\Large}{}{}
\patchcmd\@setauthors
{\MakeUppercase{\authors}}
{\authors}
{}{}
\makeatother
\newtheorem{theorem}{Theorem}[section]

\newtheorem{lemma}{Lemma}[section]
\newtheorem{definition}{Definition}[section]
\newtheorem{remark}{Remark}[section]
\newtheorem{example}{Example}[section]
\newtheorem{corollary}{Corollary}[section]
\newtheorem{Proof of Theorem}{Proof}
\newtheorem{proposition}{Proposition}[section]

\hypersetup{urlcolor=blue, citecolor=red, linkcolor= blue}

\def\#{\sharp}

\makeatletter
\renewcommand\subsection{\@startsection{subsection}{2}%
	\z@{.7\linespacing\@plus\linespacing}{.5\linespacing}%
	{\normalfont\bfseries}}
\makeatother

\begin{document}
	\title[On the $A$-$q$-Numerical Range of Operators in Semi-Hilbertian Spaces]{On the $A$-$q$-Numerical Range of Operators in Semi-Hilbertian Spaces}
 \author{Jyoti Rani}
	\address{Department of mathematics, Indian Institute of Technology Bhilai, Durg, India 491002}
	\email{jyotir@iitbhilai.ac.in}   
	
	\author{Arnab Patra}
	\address{Department of mathematics, Indian Institute of Technology Bhilai, Durg, India 491002}
	\email{arnabp@iitbhilai.ac.in}

 \author{Riddhick Birbonshi}
	\address{Department of mathematics, Jadavpur University, Kolkata 700032, West Bengal, India}
	\email{riddhick.math@gmail.com}
    
	

\begin{abstract} 
This study investigates the $A$-$q$-numerical range of an operator within the framework of semi-Hilbertian spaces. Several fundamental properties of the $A$-$q$-numerical range are established, including spectral inclusion results and a disk union formula. Bounds for the $A$-$q$-numerical radius are derived, extending and generalizing previously known results. Finally, the notion of $A$-nilpotent operator is introduced, and it is shown that the $A$-$q$-numerical range of an $A$-nilpotent operator with index $2$ is a disk (open or closed) in the complex plane.
	\end{abstract}
	
	\maketitle

\section{Introduction and Preliminaries}
Let $\mathcal{H}$ be a separable, infinite-dimensional complex Hilbert space endowed with the inner product $\langle .,. \rangle$ and $\|.\|$ be the norm induced by the inner product. The $C^*$-algebra of all bounded linear operators on $\mathcal{H}$ is denoted by $\mathcal{B}(\mathcal{H})$. Let $A \in \mathcal{B}(\mathcal{H})$ be a positive operator, i.e., $\langle Ax,x \rangle \ge 0$ for all $x \in \mathcal{H}$. For any $T \in \mathcal{B}(\mathcal{H}) $, the range and null space of $T$ are denoted by $R(T)$ and $ N(T)$ respectively. The orthogonal projection of $\mathcal{H}$ onto $\overline{{R}(A)}$ is denoted by $P_{\overline{{R}(A)}}$. The operator $A$ naturally induces a semi-inner product $\langle \cdot, \cdot \rangle_A$ on $\mathcal{H}$, given by $\langle x, y \rangle_A = \langle Ax, y \rangle$ for all $x, y \in \mathcal{H}$. The corresponding semi-norm is given by $\|x\|_A = \sqrt{\langle x, x \rangle_A}$ for all $x \in \mathcal{H}$. The space $\mathcal{H}$ equipped with this semi-inner product is referred to as a semi-Hilbertian space. The semi-inner product $\langle \cdot, \cdot \rangle_A$ induces an inner product $[\cdot, \cdot]$ on the quotient space $\mathcal{H}/N(A)$, defined by $[\bar{x}, \bar{y}] = \langle x, y \rangle_A$ for all $\bar{x}, \bar{y} \in \mathcal{H}/N(A)$. The completion of $(\mathcal{H}/N(A), [\cdot, \cdot])$ is isometrically isomorphic to the Hilbert space $\textbf{R}(A^{1/2}) := ({R}(A^{1/2}), (\cdot, \cdot))$, where the inner product $(\cdot, \cdot)$ is given by
\[
(A^{1/2} x, A^{1/2} y) = \langle P_{\overline{R(A)}} x, P_{\overline{R(A)}}y \rangle.
\]
for all $x,y \in \mathcal{H}$ \cite{feki2020tuples}.
Given $T \in \mathcal{B}(\mathcal{H})$, if there exists a constant $c > 0$ such that $\|Tx\|_A \leq c\|x\|_A$ for all $x \in \overline{{R}(A)}$, then the $A$-operator semi-norm of $T$, denoted by $\|T\|_A$, is defined as
\[
\|T\|_A = \sup_{x \in \overline{{R}(A)}, x \neq 0} \frac{\|Tx\|_A}{\|x\|_A}.
\]
Let $\mathcal{B}_{A^{1/2}}(\mathcal{H})$ be the set of all operators admitting $A^{\frac{1}{2}}$-adjoint. By Douglas Theorem \cite{douglus}, we have
\begin{equation*}
	\mathcal{B}_{A^{1/2}}(\mathcal{H}) = \{ T \in \mathcal{B}(\mathcal{H}) : \exists c > 0 \text{ such that } \|Tx\|_A \leq c\|x\|_A \ \forall x \in \mathcal{H} \}.
\end{equation*}
An operator $T \in {B}_{A^{1/2}}(\mathcal{H})$ is known as an $A$-bounded operator, and if $T \in \mathcal{B}_{A^{1/2}}(\mathcal{H})$, then $T(N(A)) \subseteq N(A)$ \cite{sen2024note}.
 An operator $T \in \mathcal{B}_{A^{1 / 2}}(\mathcal{H})$ is said to be $A$-invertible in $\mathcal{B}_{A^{1 / 2}}(\mathcal{H})$ if there exists $S \in \mathcal{B}_{A^{1 / 2}}(\mathcal{H})$ such that $A T S=A S T=A$ and $S$ is called $A$-inverse of $T$. For more details, see \cite{baklouti2022spectral}. 

For $T \in \mathcal{B}_{A^{1 / 2}}(\mathcal{H})$, $A$-point spectrum $\sigma_{A_p}(T)$, $A$-approximate point spectrum $\sigma_{A_{\text{app}}}(T)$, and $A$-spectrum $\sigma_{A}(T)$ are defined as
\begin{align*}
  \sigma_{A_p}(T) &=\left\{\lambda \in \mathbb{C} : \exists~ x (\neq 0) \in \overline{R(A)} \text{ such that } P_{\overline{R(A)}} T x = \lambda x \right\},\\
\sigma_{A_{\text{app}}}(T) &= \left\{\lambda \in \mathbb{C} : \exists \left\{x_n\right\} \subseteq \overline{R(A)} \text{ with } \left\|x_n\right\|_A = 1 \text{ such that } \left\|(T - \lambda I) x_n\right\|_A \rightarrow 0 \right\}, ~\text{and}  \\
\sigma_A(T)&=\left\{\lambda \in \mathbb{C}:(T-\lambda I) \text { is not } A \text {-invertible in } \mathcal{B}_{A^{1 / 2}}(\mathcal{H})\right\}.
\end{align*}

If $T \in \mathcal{B}_{A^\frac{1}{2}}(\mathcal{H})$, then the $A$-spectral radius of $T$ is defined as
\begin{equation}\label{srf}
 r_A(T)=\lim _{n \rightarrow \infty}\left\|T^n\right\|_A^{1 / n}   
\end{equation}
\cite[Theorem 1]{feki2020spectral}.
$
$ Observe that the Moore-Penrose inverse $A^{\dagger}$ of $A$ satisfies the following identities (see \cite{engl1981new})
$$
A^{\dagger} A = \left(A^{1 / 2}\right)^{\dagger} A^{1 / 2} = P_{\overline{R(A)}}, \quad A P_{\overline{R(A)}} = A \text{ and } A^{1 / 2} P_{\overline{R(A)}} = P_{\overline{R(A)}} A^{1 / 2} = A^{1 / 2},
$$
and hence the $A$-point spectrum of $T$ can also be expressed as
$$
\begin{aligned}
	\sigma_{A_p}(T) & = \left\{ \lambda \in \mathbb{C} : \exists~ x (\neq 0) \in \overline{R(A)} \text{ such that } A T x = \lambda A x \right\}.
\end{aligned}
$$

For $T \in \mathcal{B(H)}$, an operator $W\in \mathcal{B(H)}$ is called an $A$-adjoint of $T$ if $\langle Tx,y \rangle_A=\langle x,Wy \rangle_A$ for all $x,y \in \mathcal{H}$. The collection of all operators that possess an $A$-adjoint is denoted by $\mathcal{B}_A(\mathcal{H})$. By applying Douglas theorem, we deduce that for any $T \in \mathcal{B}(\mathcal{H})$, if $R(T^*A) \subseteq R(A)$, then $T \in \mathcal{B}_A(\mathcal{H})$, that is,
\[
\mathcal{B}_A(\mathcal{H}) = \{T \in \mathcal{B}(\mathcal{H}) : R(T^*A) \subseteq R(A)\}.
\]
Moreover, $\mathcal{B}_A(\mathcal{H})$ and $\mathcal{B}_{A^{1/2}}(\mathcal{H})$ are sub-algebras of $\mathcal{B(H)}$ and $\mathcal{B}_A(\mathcal{H})\subseteq\mathcal{B}_{A^{1/2}}(\mathcal{H}) \subseteq \mathcal{B(H)}$. Also, $\mathcal{B}_A(\mathcal{H})=\mathcal{B}_{A^{1/2}}(\mathcal{H})= \mathcal{B(H)}$, if $A$ is one-one and $R(A)$ is closed in $\mathcal{H}$.
If $T \in \mathcal{B}_A(\mathcal{H})$, then by Douglas Theorem \cite{douglus}, the reduced solution of the equation $AX = T^*A$ is denoted by $T^{\#}$, where $T^{\#} = A^\dagger T^* A$ and $R(T^{\#}) \subseteq \overline{R(A)}$. The operator $T \in \mathcal{B}_A(\mathcal{H})$ is said to be $A$-self-adjoint if $AT$ is self-adjoint, i.e., $AT = T^*A$ \cite{ahmed2012isometric}. Moreover, if $T$ is $A$-self-adjoint, it does not necessarily imply that $T=T^\#$. However, 
we have $T = T^\#$ if and only if $T$ is an $A$-self-adjoint operator and ${R}(T) \subseteq \overline{{R}(A)}$. The operator $T$ is $A$-positive and we simply write $T \geq_A 0$ if $AT$ is a positive operator. Clearly, if an operator $T$ is $A$-self-adjoint, then $T \in \mathcal{B}_A(\mathcal{H})$. Furthermore, if $T \in \mathcal{B}_A(\mathcal{H})$ is $A$-self-adjoint, then so is $T^\#$, and the property 
$
(T^\#)^{\#} = T^{\#}
$
holds.
An operator $T \in \mathcal{B}_A(\mathcal{H})$ is referred to as an $A$-normal operator if and only if $TT^{\#} = T^{\#}T$ \cite{ahmed2012isometric}. While it is well-known that all self-adjoint operators in a Hilbert space are normal, this fact may not hold for $A$-self-adjoint operators. In other words, $A$-self-adjoint operators need not necessarily be $A$-normal.
 An operator $U \in \mathcal{B}_A(\mathcal{H})$ is called $A$-unitary if 
$
\|U x\|_A = \|U^\# x\|_A = \|x\|_A \quad \text{for all } x \in \mathcal{H}.
$
Further, we have 
$
U^\# U = (U^\#) U^\# = P_{\overline{R(A)}}
$
 \cite{arias2008partial}. Note that if $U$ is $A$-unitary so is $U^\#$ and $ \| U \| _A = \| U^\# \| _A = 1$\cite{arias2008partial}. An element $x \in \mathcal{H}$ is called $A$-orthogonal to $y \in \mathcal{H}$ if $\langle x, y \rangle_A = 0$, denoted by $x \perp_A y$. Also, if $\mathcal{W}\subseteq \mathcal{H}$, then $\mathcal{W}^\perp=\{x \in\mathcal{H}:\langle x,w \rangle _A=0 ~\mbox{for all}~ w \in \mathcal{W}\}$.

The concept of the $A$-$q$-numerical range and the $A$-$q$-numerical radius was first introduced in \cite{feki2024joint}. In the sequel, we assume that $q\in \mathcal{D}$, where  $\mathcal{D}=\{z \in \mathbb{C}: |z|\le 1\}$ denotes the closed unit disk in the complex plane. For a subspace $\mathcal{V}$ of $\mathcal{H}$, we denote by $\mathbb{S}_{q, A}(\mathcal{V})$ the subset of $\mathcal{V} \times \mathcal{V}$ given by
\begin{equation*}
	\mathbb{S}_{q, A}(\mathcal{V})=\left\{ (x,y)\in \mathcal{V}\times \mathcal{V} : \Vert x \Vert_{A}=\Vert y \Vert_{A}=1 \textup{ and } \langle x, y \rangle_{A}=q \right\}.
\end{equation*}
\begin{definition}
 For $T \in \mathcal{B}_{A^{1 / 2}}(\mathcal{H})$, the $A$-$q$-numerical range $W_{q,A}(T)$ and the $A$-$q$-numerical radius $w_{q,A}(T)$ of $T$, are respectively defined as follows:
\begin{align*}
 W_{q,A}(T)&=\left\{\langle T x, y\rangle_A: (x,y) \in \mathbb{S}_{q, A}(\mathcal{H}) \right\},~\text{and}\\ 
 w_{q,A}(T)&=\sup \left\{|\lambda|: \lambda \in W_{q,A}(T)\right\} .
\end{align*} 
\end{definition}	

Feki et al. \cite{feki2024joint} proved that the $A$-$q$-numerical range $W_{q,A}(T)$ is convex. The aforementioned definition serves as a unifying framework that generalizes several significant concepts in the literature. In fact, if $q = 1$, it reduces to the $A$-numerical range \cite{zamani2019numerical, sen2024note}, if $A = I$, it yields the $q$-numerical range \cite{tsing1984constrained, stankovic2024some, kittaneh2025estimation}, and in the special case when $A = I$ and $q = 1$, it coincides with the classical numerical range.

Section 2 of this paper discusses the relationship between the $A$-$q$-numerical range and the $A$-spectrum in the setting of semi-Hilbertian structure. In this section, we derive a circular union formula for the $A$-$q$-numerical range, showing that it can be expressed as the union of circular discs, which generalizes \cite[Lemma 5]{tsing1984constrained}. An alternative expression for the $A$-$q$-numerical range and its radius is also obtained. The analysis in this section emphasizes the connection between the $q$-numerical range of operators in Hilbert and semi-Hilbertian spaces, and a complete characterization of the $A$-$q$-numerical range for $A$-self-adjoint matrices is provided. In particular, we prove that the $A$-$q$-numerical range of an $A$-self-adjoint matrix is an elliptic disk with foci at $q \lambda_1$ and $q \lambda_m$, and a minor axis of length $\sqrt{1-|q|^2}(\lambda_1 - \lambda_m)$, where $\lambda_1$ and $\lambda_m$ denote the largest and smallest $A$-eigenvalues of $T$, respectively. This result generalizes  \cite[Theorem 3.4, p. 384]{gau2021numerical}.

Section 3 introduces the notion of $A$-nilpotent operators and extends the classical class of nilpotent operators to semi-Hilbertian settings. Several key properties are established, generalizing the corresponding results for nilpotent operators. It is shown that the $A$-$q$-numerical range of an $A$-nilpotent operator $T$ of index $2$ is a disk (open or closed) centered at the origin, and $w_{q,A}(T) \le \left(\frac{1+\sqrt{1-|q|^2}}{2}\right)\|T\|_A$. In addition, an upper bound for the $A$-$q$-numerical radius of a class of $A$-nilpotent operators of index $3$ is derived.

\section{Some Results on $A$-$q$-numerical range }
We start this section with the following lemma, which ensures the existence of a vector $y \in \mathcal{H}$ corresponding to any $x \in \mathcal{H}$ with $\|x\|_A = 1$. Moreover, when $\dim R(A) = 1$, it has been shown that $W_{q,A}(T) \neq \emptyset$ if and only if $|q| = 1$ \cite{feki2024joint}. Therefore, in what follows, we consider the nontrivial case $\dim R(A) \geq 2$.

\begin{lemma}\label{flemma}
Let $T \in \mathcal{B}_{A^{\frac{1}{2}}}(\mathcal{H})$, where $A \in \mathcal{B}(\mathcal{H})^+$ such that $\dim(R(A)) \geq 2$, and $q \in \mathcal{D}$. If $x \in \mathcal{H}$ with $\|x\|_A = 1$, then there exists $z \in \mathcal{H}$ such that $\langle x, z \rangle_A = 0$ and $\|z\|_A = 1$.In particular, if $x \in \mathcal{H}$ with $\|x\|_{A} = 1$, then there exists a vector $y = \overline{q}x + \sqrt{1 - |q|^{2}}z$ such that $\|y\|_{A} = 1$ and $\langle x, y \rangle_{A} = q$. 
\end{lemma}
\begin{proof}
 Let $ \mathcal{H}=span\{Ax\} \oplus (span\{Ax\})^\perp$. If there does not exist any non-zero $z'$ such that $\langle x,z' \rangle_A=\langle Ax,z' \rangle=0$, then $(span\{Ax\})^\perp=\{0\}$. This implies $\mathcal{H}=span\{Ax\}$ and $\dim(\mathcal{H})=1$, which is not possible as $\dim(R(A)) \geq 2$. Thus, there exists a non-zero $z'$ such that $\langle x,z' \rangle_A=\langle Ax,z' \rangle=0$. If all such $z' \in N(A)$, then $(span\{Ax\})^\perp \subseteq N(A)$. Also, $\mathcal{H}={N(A)} \oplus \overline{R(A)}$. Thus, we have $\overline{R(A)}\subseteq span\{Ax\}$. Therefore, $\dim (\overline{R(A)}) \le 1$, which is a contradiction. Thus for any $x \in\mathcal{H}$ with $\|x\|_A=1$, there exists a $z=\frac{z'}{\|z\|_A}$ such that $\langle x,z \rangle_A=0$ and $\|z\|_A=1$.   \end{proof}

\begin{remark}
 It is important to note that if $\dim(R(A)) = 1$ and $x \in \mathcal{H}$ with $\|x\|_A = 1$, then there need not exist a vector $z \in \mathcal{H}$ such that $\langle x, z \rangle_A = 0$ and $\|z\|_A = 1$.
 For example, Let $A=\begin{pmatrix}
     1 & 0 & 0 \\
     0 & 0 & 0 \\
      0 & 0 & 0
 \end{pmatrix}$ and $x=\begin{pmatrix}
     i \\
     1\\
     0
 \end{pmatrix} \in \mathbb{C}^3$. Clearly $\|x\|_A=1$. If there exists $ z \in \mathbb{C}^3 $ such that $ \langle x, z \rangle_A = 0 $, then it is easy to verify that $ z \in N(A)$. This implies $ \|z\|_A = 0 \ne 1 $. 
\end{remark}
\begin{remark}\label{recon}
 For any $y \in \mathcal{H}$ with $\|y\|_A=1$ and $\langle x,y \rangle_A=q$, set $z=\frac{1}{\sqrt{1-|q|^2}}(y-\overline{q}x)$ with $|q| < 1$, resulting in $\langle x,z \rangle_A=0$ and $\|z\|_A=1$. Thus, there exists a one-to-one correspondence between such a $z$ and $y$ mentioned in Lemma \ref{flemma}. If $|q| = 1$, then we have 
$|q| = |\langle x, y \rangle_A| = \|x\|_A \|y\|_A = 1$. 
This implies 
$|\langle A^{\frac{1}{2}} x, A^{\frac{1}{2}} y \rangle| = \|A^{\frac{1}{2}} x\| \, \|A^{\frac{1}{2}} y\|$, 
which gives equality in the Cauchy--Schwarz inequality. 
Consequently, 
$A^{\frac{1}{2}} x = \lambda A^{\frac{1}{2}} y$. Then $\lambda={q}$. In this case, $W_{q,A}(T)={q}W_A(T)$.
\end{remark}
In the following result, we establish some spectral inclusion relations of $W_{q, A}(T)$.

\begin{proposition}\label{spectrum}
	Let $T \in \mathcal{B}_{A^\frac{1}{2}}(\mathcal{H})$ and $q \in \mathcal{D}$. Then
	\begin{itemize}
		\item [(i)] $q \sigma_{A_p}(T) \subseteq W_{q,A}(T)$,
		\item [(ii)] $q\sigma_{A_{app}}({T})\subseteq \overline{W_{q,A}(T)},$
		where $\overline{W_{q,A}(T)}$ denotes the closure of $W_{q,A}(T)$,
	\end{itemize}
 
 \begin{itemize}
     \item [(iii)] If $R(A)$ is closed, then $q \sigma_A(T)\subseteq \overline{W_{q,A}(T)}$.
 \end{itemize}
	
\end{proposition}
\begin{proof}
\begin{itemize}
    \item [(i)]	Let $\lambda \in  \sigma_{A_p}(T).$ Then, there exists $x\in \overline{R(A)}$, $\|x\|_A=1$ such that $ATx=\lambda Ax.$ Using Lemma \ref{flemma} there exists a vector $y= \overline{q}x+ \sqrt{1-|q|^2}z$, where $z\in \mathcal{H}$, and $\Vert z \Vert_A =1$ and $\langle x, z \rangle_A=0.$ Clearly, $(x,y)\in \mathbb{S}_{q,A}(\mathcal{H}).$  Moreover,
\begin{align*}
 	\lambda q=  &\lambda \langle x,y \rangle_A \\
   =& \langle   \lambda Ax,y \rangle  \\
   =& \langle   ATx,y \rangle  \\
   = &\langle Tx,y \rangle_A.   
\end{align*}

		This implies, $\lambda q \in 	W_{q,A}(T)$ and consequently,  $q  \sigma_{A_p}(T) \subseteq W_{q,A}(T)$.
 \item [(ii)]
Let $\lambda \in \sigma_{A_{app}}({T}),$ then there exists a sequence $\{x_n\}$ in $\overline{R(A)}$ with $\|x_n\|_A=1$ such that 
		$
			\lim_{n\to \infty }\big\|( T-\lambda I)x_n \big\|_A =0.
		$	
		On the other hand, for each $x_n \in \overline{R(A)}$, there exists $y_n \in \mathcal{H}$ such that $(x_n, y_n) \in \mathbb{S}_{q,A}(\mathcal{H})$. Hence,  $\langle Tx_n,y_n \rangle_A \in W_{q,A}(T).$ Using Cauchy-Schwarz's inequality, we get
\begin{align*}
 	|\langle Tx_n,y_n \rangle_A -q\lambda|= &|\langle Tx_n,y_n \rangle_A -\lambda \langle x_n,y_n \rangle_A|\\
  =& |\langle (T-\lambda I) x_n,y_n \rangle_A|\\  
  \le &\|(T-\lambda I)x_n\|_A\|y\|_A\\
= &\|(T-\lambda I)x_n\|_A,   
\end{align*}

		for all $n\in \mathbb{N}.$	This implies,
		$$
			\lim_{n\to \infty} \langle Tx_n,y_n \rangle_A = q \lambda.
		$$
		Therefore $q\lambda \in\overline{W_{q,A}(T)},$ and hence $q\sigma_{A_{app}}({T})\subseteq \overline{W_{q,A}(T)}.$  
\item [(iii)]      Since $\partial \sigma_A(T)\subseteq \sigma_{A_{app}}(T)$ \cite{sen2024note}, we have
	$
	q\partial \sigma_A(T)\subseteq q\sigma_{A_{app}}(T)\subseteq \overline{W_{q,A}(T)}.
	$
Convexity of $W_{q,A}(T)$\cite{feki2024joint} implies that
	$q \sigma_A(T)\subseteq \overline{W_{q,A}(T)}$. 
\end{itemize}       
\end{proof}

\begin{remark}
 If $T(N(A)) \subseteq N(A)$ and $\mathcal{H}$ is finite dimensional, then by Theorem 26\cite{feki2024joint}, we have $q\sigma_{A}(T)\subseteq W_{q,A}(T)$, where $R(A)$ is closed.
\end{remark}
We now present an example demonstrating that the closedness of $R(A)$ is not a necessary condition in Proposition \ref{spectrum}(iii).
\begin{example}
 Let $\{e_n\}$ be the standard orthonormal basis of the Hilbert space $\ell^2$. For a fixed $p \in \mathbb{N}$ with $p \geq 2$, define a bounded linear operator $A : \ell^2 \to \ell^2$ by
\[
Ae_n = \frac{1}{(\log (n+1))^p} e_n, \quad \forall n \in \mathbb{N}.
\]
Clearly, $A$ is a positive operator and $R(A)$ is not closed. Let $\{\lambda_n\}$ be a sequence of complex scalars, where $|\lambda_n| \le 1$ and $\lambda_n \to \lambda_0$. Define a bounded linear operator $T : \ell^2 \to \ell^2$ by
\[
Te_n = \lambda_n e_n, \quad \forall n \in \mathbb{N}.
\]

It is straightforward to verify that
\[
\{\lambda_n : n \in \mathbb{N}\}\subseteq \sigma_{A_p}(T).
\]
Since $\sigma_{A_{app}}(T)$ is a closed set \cite[Theorem 2.5]{majumdar2024approximate} and $\sigma_{A_p}(T)\subseteq \sigma_{A_{app}}(T)$, we have 
\[
\{\lambda_n : n \in \mathbb{N} \cup \{0\}\}\subseteq \sigma_{A_{app}}(T).
\]

Now, assume $\gamma \notin \{\lambda_n : n \in \mathbb{N} \cup \{0\}\}$. Then there exists $\delta > 0$ such that $|\gamma - \lambda_n| \geq \delta$ for all $n \in \mathbb{N} \cup \{0\}$. Let $\{{x'}_i\}$ be an arbitrary sequence in $\ell^2$ such that ${x'}_i = \sum_{n=1}^\infty x_n^{(i)}e_n$ and
\[
\|{x'}_i\|_A = \left( \sum_{n=1}^\infty |x_n^{(i)}|^2 \|e_n\|_A^2 \right)^{1/2} = 1.
\]
Then,
\[
\|(T - \gamma I){x'}_i\|_A^2 = \sum_{n=1}^\infty |\lambda_n - \gamma|^2 |x_n^{(i)}|^2 \|e_n\|_A^2 \geq \delta^2 > 0.
\]
Thus, $\gamma \notin \sigma_{A_{app}}(T)$ and this implies $\gamma \notin \sigma_{A_{p}}(T)$. 
Therefore,
\begin{eqnarray*}
\sigma_{A_{p}}(T) = \{\lambda_n : n \in \mathbb{N}\}\\
    \sigma_{A_{app}}(T) = \{\lambda_n : n \in \mathbb{N} \cup \{0\}\}
\end{eqnarray*}
As $\partial \sigma_A(T) \subseteq \sigma_{A_{app}}(T) $ \cite[Remark 2.21]{majumdar2024approximate}, we have boundary of $\sigma_A(T)$ is countable. Moreover, $\sigma_A(T)$ is a compact set \cite[Proposition 5.8]{baklouti2022spectral}. We conclude that $\sigma_A(T)=\partial \sigma_A(T)$. Also $\sigma_A(T)=\partial \sigma_A(T) \subseteq \sigma_{A_{app}}(T) \subseteq \sigma_A(T)$, yields 
\[
\sigma_A(T) = \partial \sigma_A(T) = \{\lambda_n : n \in \mathbb{N} \cup \{0\}\}.
\]
Using Proposition \ref{spectrum}(i), we have $\{q\lambda_n : n \in \mathbb{N}\} \subseteq W_{q,A}(T)$. Hence, $\{q\lambda_n : n \in \mathbb{N} \cup \{0\}\} =q \sigma_A(T)\subseteq \overline{W_{q,A}(T)}$.
\end{example}


Next, to prove our forthcoming result, we require the following lemma, which establishes the $A$-unitary invariance of the $A$-$q$-numerical range, i.e., \( W_{q,A}(UTU^\#) = W_{q,A}(T) \).
\begin{lemma}\label{propuni}
Let $T \in \mathcal{B}_{A^\frac{1}{2}}(\mathcal{H})$ and $q \in \mathcal{D}$. Then $W_{q,A}(UTU^\#)=W_{q,A}(T)$, where $U$ is $A$-unitary operator. 
\end{lemma}
\begin{proof}
Let $\lambda \in W_{q,A}(T)$. Then, there exists $(x,y) \in \mathbb{S}_{q, A}(\mathcal{H})$ such that $\lambda=\langle Tx,y \rangle _A$. Using decomposition $\mathcal{H}=\overline{R(A)} \oplus N(A) $, we can write $x=P_{\overline{{R}(A)}}x+y'$ and $y=P_{\overline{{R}(A)}}y+y''$, where $y', y'' \in N(A)$. Now
\begin{align*}
	\langle Tx,y \rangle _A = &\langle T(P_{\overline{{R}(A)}}x+y'),P_{\overline{{R}(A)}}y+y'' \rangle _A\\
	= & \langle TP_{\overline{{R}(A)}}x,P_{\overline{{R}(A)}}y \rangle_A.
\end{align*}
As $U$ is $A$-unitary, we can take $P_{\overline{{R}(A)}}=U^\#U=(U^\#)^\#U^\#$. Therefore,
$\langle Tx,y \rangle _A =\langle TU^\#Ux,U^\#Uy \rangle_A=\langle UTU^\#Ux,Uy \rangle_A,$ 
where $\|Ux\|_A=\|x\|_A=1$ and $\|Uy\|_A=\|y\|_A=1$. Moreover, 
\begin{align*}
 \langle Ux,Uy \rangle_A=&\langle U^\#Ux,y \rangle_A\\
 =&\langle P_{\overline{{R}(A)}}x,y \rangle_A=\langle x-y',y \rangle_{A}\\
 =&\langle x,y \rangle_{A}-\langle y',y \rangle_{A}=q.  
\end{align*}

Thus, $\lambda \in W_{q,A}(UTU^\#)$ and $W_{q,A}(T) \subseteq W_{q,A}(UTU^\#)$. 

Conversely, let $\lambda \in W_{q,A}(UTU^\#)$. There exist $x,y \in \mathbb{S}_{q, A}(\mathcal{H})$ such that $\lambda = \langle UTU^\# x,y \rangle_{A}$. Thus,
\begin{equation*}
    \langle UTU^\# x,y \rangle_{A}=\langle TU^\# x,U^\# y \rangle_{A}. 
\end{equation*}
Since $U$ is $A$-unitary, it follows that
$\|U^\#x\|_A=\|x\|_A=1, \ \|U^\#y\|_A=\|y\|_A=1,$ and 
\begin{align*}
  \langle U^\#x,U^\#y\rangle_{A}=&\langle (U^\#)^\#U^\#x,y\rangle_{A}\\
  =& \langle P_{\overline{{R}(A)}}x,y \rangle_{A}=\langle x-y',y\rangle_{A}\\
  =&q.  
\end{align*}

Thus, $\lambda \in W_{q,A}(T)$ and $W_{q,A}(UTU^\#)\subseteq W_{q,A}(T)$. This completes the proof.
\end{proof}
Next, we extend the classical notion of approximate unitary equivalence of two operators \cite{luketero2024approximately} from the Hilbert–space setting to the semi-Hilbertian framework.
\begin{definition}
 Two operators $T,S \in \mathcal{B}_{A^\frac{1}{2}}(\mathcal{H})$ are said to be $A$-approximately unitarily equivalent if there exists a sequence $\{U_n\}_{n=1}^\infty$ of $A$-unitary operators such that 
 \begin{equation*}
     \lim_{n \to \infty}\|U_nSU_n^\#-T\|_A=0.
 \end{equation*}
\end{definition}
The following theorem establishes a relation between the $A$-$q$-numerical range and radii of two $A$-approximately unitarily equivalent operators.
\begin{theorem}
Let $q \in \mathcal{D}$ and $T,S \in \mathcal{B}_{A^\frac{1}{2}}(\mathcal{H})$ are $A$-approximately unitarily equivalent operators. Then $\overline{W_{q,A}(T)}=\overline{W_{q,A}(S)}$ and $w_{q,A}(T)=w_{q,A}(S)$. In this case, $ q\sigma_A(T) \subseteq \overline{ W_{q,A}(S)}$, if $R(A)$ is closed.    
\end{theorem}
\begin{proof}
If $\lambda \in W_{q,A}(T)$, there exist $(x,y) \in \mathbb{S}_{q, A}(\mathcal{H})$ such that $\lambda=\langle Tx, y \rangle_A$. Let $\{U_n\}_{n=1}^\infty$ is a sequence of $A$-unitary operators. Using Lemma \ref{propuni}, we have $\langle U_nSU_n^\# x,y \rangle_A \in W_q(S)$. Therefore,
\begin{align*}
    \lim_{n \to \infty}| \lambda -\langle U_nSU_n^\# x,y \rangle_A|= &\lim_{n \to \infty}| \langle (T-U_nSU_n^\#)x,y \rangle_A |\\
    \le & \lim_{n \to \infty}\| T-U_nSU_n^\#\|_A.
\end{align*}
Since $S$ and $T$ are $A$-approximately unitarily equivalent operators, we have $\lim_{n \to \infty}\| T-U_nSU_n^\#\|=0$ and $\lambda \in \overline{ W_{q,A}(S)}$. Thus, $\overline{ W_{q,A}(T)} \subseteq \overline{ W_{q,A}(S)}$. By interchanging $T$ and $S$, we get $\overline{ W_{q,A}(S)}\subseteq \overline{ W_{q,A}(T)}$. Therefore, $\overline{ W_{q,A}(T)}=\overline{ W_{q,A}(S)}$. Now, 
\begin{align*}
    w_{q,A}(T)= &w_{q,A}(T-U_nSU_n^\#+U_nSU_n^\#)\\
             \le & w_{q,A}(T-U_nSU_n^\#)+w_{q,A}(U_nSU_n^\#)\\
             \le & \|T-U_nSU_n^\#\|_A+w_{q,A}(S)\quad \text{(Using Lemma \ref{propuni})}.
\end{align*}
As $T$ and $S$ are $A$-approximately unitarily equivalent, $w_{q,A}(T) \le w_{q,A}(S)$ when $ n \to \infty.$ By interchanging $T$ and $S$, we get $w_{q,A}(S) \le w_{q,A}(T)$. Therefore, $w_{q,A}(T)=w_{q,A}(S)$. If $R(A)$ is closed, Proposition \ref{spectrum}(iii) yields, $ q\sigma(T) \subseteq \overline{ W_{q,A}(S)}$. 
\end{proof}
Considerable progress has been made in understanding the $q$-numerical range following the foundational work of N. K. Tsing, who established a key result known as the circular union formula for $W_q(T)$ ($q \in \mathcal{D}$) in \cite{tsing1984constrained}. This formula expresses the $q$-numerical range as
\begin{equation*}
 W_q(T)=\cup_x \left\lbrace z \in \mathbb{C}: \left| z-q \langle Tx,x \rangle \right| \le \sqrt{1-|q|^2}\left( \|Tx\|^2- \left| \langle Tx, x \rangle \right|^2 \right)^\frac{1}{2} \right\rbrace,   
\end{equation*}
where $x$ runs over all unit vectors. This characterization enabled Tsing to derive several important properties of $W_q(T)$, including its convexity.
Subsequently, M. T. Chien and H. Nakazato reformulated this expression into a more structured framework using the Davis-Wielandt shell, yielding the formula
\begin{equation}\label{suf}
W_q(T)= \cup_{z \in W(T)}\left\lbrace \xi \in \mathbb{C}: | \xi -qz| \le \sqrt{1-q^2}\sqrt{h(z)-|z|^2}\right \rbrace,     
\end{equation}
where $q \in [0,1]$, $h(z)=\sup\left \lbrace t \in \mathbb{R} : (z,t) \in DW(T) \right \rbrace$, and $DW(T)$ denotes the joint numerical range of $(T, T^*T)$, also known as the Davis-Wielandt shell of $T$. Within this framework, they investigated the $q$-numerical radius of weighted shift operators with periodic weights \cite{chien2007q} and explored the relationship between $q$-numerical range and the Davis-Wielandt shell \cite{chien2002davis}. In another notable contribution, H. Nakazato examined the boundary structure of $W_q(T)$ for normal matrices $T$ using the formula \eqref{suf} \cite{nakazato1995boundary}.
In our forthcoming result, we extend this circular union formula for the $A$-$q$-numerical range in the setting of semi-Hilbertian space.

\begin{theorem}\label{circf}
	Let $T \in \mathcal{B}_{A^\frac{1}{2}}(\mathcal{H})$, $\dim(R(A)) \ge 3$, and $q \in \mathcal{D}$. Then
	\[W_{q,A}(T)=\cup \{\mathcal{G}(x): x \in \mathcal{H}, \|x\|_A=1\},~ \] 
where $\mathcal{G}(x)=\{a \in\mathbb{C}: |a-q\langle Tx,x \rangle_A | \le \alpha \sqrt{1-|q|^2} \}$	
and $\alpha=\left(\|Tx\|_A^2-|\langle Tx,x \rangle_A|^2\right)^\frac{1}{2}.$
\end{theorem}

\begin{proof}
Considering Proposition 20(i) \cite{feki2024joint}, it is sufficient to examine the case where $0 \le q \le 1$. 
Let $x \in \mathcal{H}$ be any vector with $\|x\|_A=1$ and
	$$\mathcal{G}_1(x)=\{\langle {T}x,y \rangle_A: y \in \mathcal{H},\|y\|_A=1, \langle x,y \rangle_A=q\}.$$ Then,
	$
	W_{q,A}(T)=\cup \{\mathcal{G}_1(x): x \in \mathcal{H},\|x\|_A=1\}.$
	Our task is to prove that
	\begin{equation*}
		\mathcal{G}(x)= \mathcal{G}_1(x).
	\end{equation*}
If $q=1$, then similar to Remark \ref{recon}, we have $A^\frac{1}{2}x=A^\frac{1}{2}y$. Clearly, $\mathcal{G}(x)= \mathcal{G}_1(x).$    
 If $\alpha=0$, then $\mathcal{G}(x)=\{q\langle Tx,x \rangle_A\}$. Also,
\begin{eqnarray*}
 & &| \langle Tx,x \rangle_A|= \|Tx\|_A \\
& \Leftrightarrow & | \langle ATx,x \rangle|= \|A^\frac{1}{2}Tx\| \\
& \Leftrightarrow & | \langle A^\frac{1}{2}Tx,A^\frac{1}{2}x \rangle|= \|A^\frac{1}{2}Tx\|.
\end{eqnarray*}
The last equation implies the equality case in the Cauchy-Schwarz inequality. Thus, $A^\frac{1}{2}Tx=\lambda A^\frac{1}{2}x,$ for some $\lambda \in \mathbb{C}.$ Therefore, $\lambda=\langle Tx,x \rangle_A$ and
\[ \langle Tx,y \rangle_A= \langle A^{\frac{1}{2}}Tx,A^{\frac{1}{2}}y \rangle=\langle \lambda A^{\frac{1}{2}} x, A^{\frac{1}{2}}y \rangle=\lambda \langle x,y \rangle_A=q\lambda=q \langle Tx,x \rangle_A.\]
Thus, $\mathcal{G}_1(x)=\{q\langle Tx,x \rangle_A\}$. Therefore, $\mathcal{G}(x)=\mathcal{G}_1(x)$.
	
We may assume $q<1$ and $\alpha \neq 0$.  
 Let $\langle {T}x,y \rangle_A \in \mathcal{G}_1(x).$
 Then $(x,y) \in \mathbb{S}_{q, A}(\mathcal{H})$ and using Remark \ref{recon}, $y$ can be expressed as $y=qx+ \sqrt{1-q^2}z$, where $\|z\|_A=1$ and $\langle x,z \rangle_A =0$. 
Take $v=\frac{Tx-\langle Tx,x\rangle_Ax}{\alpha}$ such that $\|v\|_A=1$ and $v$ is $A$-orthogonal to $x$. Thus, ${T}x=\langle {T}x,x\rangle_A x+\alpha v$.
	In this setting, by using Cauchy-Schwarz's inequality for semi-Hilbertian spaces, we have
\begin{align*}
		|\langle {T}x,y \rangle_A-q\langle {T}x,x \rangle_A|=&|\langle {T}x,qx+ \sqrt{1-q^2}z \rangle_A-q\langle {T}x,x \rangle_A|\\
		= & \sqrt{1-q^2}|\langle {T}x,z \rangle_A|\\
		\le & \sqrt{1-q^2} \left| \langle \langle Tx,x \rangle_Ax+\alpha v ,z \rangle_A \right|\\
        = & \alpha \sqrt{1-q^2} |\langle v,z \rangle _A|\\
        \le &\alpha \sqrt{1-q^2} \|v\|_A\|z\|_A\\
        =&\alpha \sqrt{1-q^2}.
\end{align*}
	Thus, $\langle Tx, y \rangle_A \in \mathcal{G}(x)$. Therefore, $ \mathcal{G}_1(x) \subseteq \mathcal{G}(x).$ Conversely, let $a\in \mathcal{G}(x) $, $b=a-q \langle Tx,x\rangle_A$, and $\mathcal{W}=span\{x,v\}$, where $v=\frac{Tx-\langle Tx,x\rangle_Ax}{\alpha}$. Since $\dim(R(A)) \ge 3$, there exists $x_1 \in \mathcal{H}$ such that $Ax_1 \notin span\{Ax,Av\}$.
    Consider $w'=x_1-\langle x_1, x \rangle_A x-\langle x_1, v \rangle_A v$ and $w=\frac{w'}{\|w'\|_A}$. Then $w \in \mathcal{W}^{\perp_A}$ and $\|w\|_A=1$.
Also, consider $z=\frac{\overline{b}}{\alpha\sqrt{1-q^2}}v+ \sqrt{1-\left|\frac{\overline{b}}{\alpha\sqrt{1-q^2}}\right|^2}w$ and $y=qx+\sqrt{1-q^2}z$. Since $z$ is a vector that satisfies $\langle x,z \rangle_A =0$ and $\|z\|_A=1$, we have $\|y\|_A=1$ and $\langle x,y \rangle_A =q$. Hence,
	\begin{align*}
		\langle T x,y \rangle_A -q 	\langle T x,x \rangle_A=&\sqrt{1-q^2}	\langle T x,z \rangle_A\\
		=&\sqrt{1-q^2} \left \langle \langle {T}x,x\rangle_A x+\alpha v,~\frac{\overline{b}}{\alpha\sqrt{1-q^2}}v+ \sqrt{1-\left|\frac{\overline{b}}{\alpha \sqrt{1-q^2}}\right|^2}w \right \rangle_A\\
		=&\sqrt{1-q^2} \left\langle \alpha v,~\frac{\overline{b}}{\alpha \sqrt{1-q^2}}v \right \rangle_A\\
		=& b.
	\end{align*}
	Thus, $\langle T x,y \rangle_A-q\langle T x,x \rangle_A=b$ and $ a =\langle T x,y \rangle_A$ is in $ \mathcal{G}_1(x)$. 	Therefore, $ \mathcal{G}(x) \subseteq \mathcal{G}_1(x).$ Hence, the required result holds.
\end{proof}	
\begin{remark}
 As a direct consequence of the aforementioned theorem, we obtain that
\[
    W_{q,A}(T) = q W_A(T) \quad \text{if } \alpha = 0.
\]
\end{remark}
The following result proves that $W_{0,A}(T)$ is an open(or closed) disk centered at the origin. 
\begin{corollary}\label{corcir}
Let $T \in \mathcal{B}_{A^\frac{1}{2}}(\mathcal{H})$. Then
 $$W_{0,A}(T)=\cup_{\{x \in \mathcal{H}:\|x\|_A=1\}}\{a \in \mathbb{C}: |a| \le \left(\|Tx\|_A^2-|\langle Tx,x \rangle_A|^2\right)^\frac{1}{2}\}.$$   
\end{corollary}
\begin{proof}
 If $q=0$ in the above theorem, we have $\mathcal{G}(x)=\{ a \in \mathbb{C}: |a| \le \alpha\}$. Thus, the required result is obvious.   
\end{proof}

\begin{remark}
    In particular, for $A=I$ in Corollary \ref{corcir}, we have 
    \[W_{0}(T)=\cup_{\{x \in \mathcal{H}:\|x\|=1\}}\{a \in \mathbb{C}: |a| \le \left(\|Tx\|^2-|\langle Tx,x \rangle_A|\right)^\frac{1}{2}\},\]
    this is obtained in \cite{stolov1979hausdorff}.
\end{remark}

The following corollary gives a relation between the sets $W_A(T)$ and $W_{q,A}(T)$.
\begin{corollary}
Let $T \in \mathcal{B}_{A^{\frac{1}{2}}}(\mathcal{H})$, $\dim(R(A)) \geq 3$, and $q \in \mathcal{D}$. Then
$$q W_A(T) \subseteq W_{q,A}(T)$$. 
\end{corollary}
\begin{proof}
    If $|q| = 1$, then using Remark \ref{recon}, we have 
$W_{q,A}(T) = q W_A(T)$.
Moreover, Theorem~\ref{circf} implies 
$q W_A(T) \subseteq W_{q,A}(T)$. 
\end{proof}

In the following result, we present an alternative formulation of the $A$-$q$-numerical range and $A$-$q$-numerical radius, by employing a similar approach as \cite[Lemma 3.1]{stankovic2024some}.
\begin{proposition}\label{exp}
 For any $T \in \mathcal{B}_{A^\frac{1}{2}}(\mathcal{H})$ and $q \in \mathcal{D}$, we have
\begin{itemize}
    \item [(a)]  $W_{q,A}(T) = \left\{ q \langle Tx, x \rangle_A + \sqrt{1 - |q|^2} \langle Tx, z \rangle_A : \|x\|_A = \|z\|_A = 1, \, \langle x, z \rangle_A = 0 \right\},$

\item[(b)] $w_{q,A}(T) = \sup \left\{ |q|  |\langle Tx, x \rangle_A| + \sqrt{1 - |q|^2} |\langle Tx, z \rangle_A| : \|x\|_A = \|z\|_A = 1, \, \langle x, z \rangle_A = 0 \right\}.$
\end{itemize}  
\end{proposition}
\begin{proof}
If \( |q| = 1 \), then by Remark \ref{recon}, the result follows immediately.
Let $|q|< 1$ and $
(x,y) \in \mathbb{S}_{q, A}(\mathcal{H})$. Using $y =\overline{q}x+\sqrt{1-|q|^2}z$, with $\|z\|_A=1$ and $\langle x,z \rangle_A=0$, we can obtain the part $(a)$ easily. Also, part $(a)$ implies, 
\[
w_{q,A}(T) = \sup \left\{ \left| q \langle Tx, x \rangle_A  + \sqrt{1 - |q|^2} \langle Tx, z \rangle_A \right| : \|x\|_A = \|z\|_A = 1, \, \langle x, z \rangle_A = 0 \right\}.
\]

Take
\[
\mu = \sup \left\{ |q|  |\langle Tx, x \rangle_A| + \sqrt{1 - |q|^2}  |\langle Tx, z \rangle_A| : \|x\|_A = \|z\|_A = 1, \, \langle x, z \rangle_A = 0 \right\}.
\]
Clearly, $w_{q,A}(T) \leq \mu$. Let $\epsilon > 0$ be any arbitrary number. Then, there exist $x, z \in \mathcal{H}$ with $\|x\|_A = \|z\|_A = 1$, $\langle x, z \rangle_A = 0$ such that
\begin{eqnarray*}
 |q|  |\langle Tx, x \rangle_A| + \sqrt{1 - |q|^2}  |\langle Tx, z \rangle_A| > \mu - \epsilon,  \\
 |q \langle Tx, x \rangle_A| + \sqrt{1 - |q|^2}  |\langle Tx, z \rangle_A| > \mu - \epsilon.
\end{eqnarray*}

Considering the case of equality in the triangle inequality of complex numbers, we can choose $\theta \in \mathbb{R}$ such that the vector $z_1 = e^{i\theta} z \in \mathcal{H}$. Hence
\[
|q \langle Tx, x \rangle_A| + \sqrt{1 - |q|^2} |\langle Tx, z_1 \rangle_A| = \left| q \langle Tx, x \rangle_A + \sqrt{1 - |q|^2} \langle Tx, z_1 \rangle_A \right|.
\]
Finally, 
\begin{align*}
 w_{q,A}(T) \geq & \left| q \langle Tx, x \rangle_A + \sqrt{1 - |q|^2} \langle Tx, z_1 \rangle_A \right|\\
 =& |q \langle Tx, x \rangle_A| + \sqrt{1 - |q|^2} |\langle Tx, z_1 \rangle_A|\\
 >&\mu - \epsilon.   
\end{align*}

Thus, $w_{q,A}(T) \geq \mu $.
\end{proof}
In the forthcoming result, we obtain some relations between $w_{q, A}(T)$ and $\|T\|_A$.
\begin{corollary}\label{maincor}
    Let $q \in \mathcal{D}$. Then 
 \begin{itemize}
     \item [(i)] $|q| w_A(T) \le w_{q,A}(T)$ for $T \in \mathcal{B}_{A^\frac{1}{2}}(\mathcal{H})$,
     \item [(ii)] $\frac{|q|}{2} \|T\|_A\le w_{q,A}(T) \le \|T\|_A$ for $T \in \mathcal{B}_{A}(\mathcal{H})$,
 \end{itemize}     
 If $T$ is $A$-self-adjoint operator, then
 \begin{itemize}
     \item [(iii)] $|q| w_A(T) \le w_{q,A}(T) \le w_A(T)$ or $|q| \|T\|_A \le w_{q,A}(T) \le \|T\|_A$ for $T \in \mathcal{B}_{A}(\mathcal{H})$.
 \end{itemize}
\end{corollary}
\begin{proof}
 Proposition \ref{exp}$(a)$ implies, $|q| |\langle Tx, x \rangle_A| \le |q| |\langle Tx, x \rangle_A| + \sqrt{1 - |q|^2}  |\langle Tx, z \rangle_A| \le w_{q,A}(T)$. Thus, we have $|q|w_A(T) \le w_{q,A}(T)$. Clearly, $w_{q,A}(T) \le \|T\|_A$. Using this relation and Corollary 2.8 \cite{zamani2019numerical}, we obtain $\frac{|q|}{2} \|T\|_A\le w_{q,A}(T) \le \|T\|_A$. Moreover, if $T$ is $A$-self-adjoint operator, we have $w_A(T)=\|T\|_A$ \cite{zamani2019numerical}. Therefore, $|q| \|T\|_A \le w_{q,A}(T) \le \|T\|_A.$   
\end{proof}
	
\begin{remark}
Corollary \ref{maincor} unifies several existing results in the literature. 
\begin{itemize}
\item [(i)] For $q=1$ and $A=I$, Corollary \ref{maincor}(ii) gives the following well-known inequality for $T \in \mathcal{B(H)}$  
\[\frac{\|T\|}{2}\le w(T) \le \|T\|.\] 
If $T$ is self-adjoint Corollary \ref{maincor}(iii) implies, $w(T)=\|T\|$, which is mentioned in \cite{kittaneh2003numerical}.

\item[(ii)] For $q=1$ and $T \in \mathcal{B}_{A^\frac{1}{2}}(\mathcal{H})$, from Corollary \ref{maincor}(ii) we have \[
        \frac{1}{2}\|T\|_A \le w_{A}(T) \le \|T\|_A,
\]
 which is obtained in \cite[Corollary 2.8]{zamani2019numerical}. If $T$ is $A$-self adjoint then from Corollary \ref{maincor}(iii), we have $w_A(T)=\|T\|_A$ \cite[Lemma 2.1]{zamani2019numerical}.

 \item [(iii)]For $A=I$ and $T \in \mathcal{B(H)}$, from Corollary \ref{maincor}(i) and Corollary \ref{maincor} (ii), respectively, we have
 \begin{align*}
 &|q|w(T) \le w_q(T)\\
&\frac{|q|}{2}\|T\|\le w_{q}(T)\le \|T\| 
 \end{align*}
 which are obtained in \cite[inequality (5)]{stankovic2024some} and \cite[inequality (6)]{stankovic2024some} respectively. If $T$ is $A$-self adjoint, then from Corollary \ref{maincor}(iii), we have 
 \begin{equation*}
     |q|\|T\| \le w_q(T) \le \|T\|
 \end{equation*}
 which is the extension of \cite[inequality (7)]{stankovic2024some} for self-adjoint operators.
\end{itemize}	 
\end{remark}


\begin{corollary}
 Let $T \in \mathcal{B}_{A}(\mathcal{H})$ and $q \in \mathcal{D}\setminus \{0\}$. Then
 $
  \lim_{n \to \infty}\left(w_{q,A}(T^n)\right)^\frac{1}{n}=r_A(T).   
 $
\end{corollary}
\begin{proof}
    Since, $T \in \mathcal{B}_A(\mathcal{H})$ implies, $T^n \in \mathcal{B}_A(\mathcal{H})$. From Corollary \ref{maincor}, we have
 \begin{equation*}
 \frac{|q|}{2} \|T^n\|_A\le w_{q,A}(T^n) \le \|T^n\|_A.
 \end{equation*}
 This implies,
 \begin{equation*}
 \lim_{ n\to \infty}\left(\frac{|q|}{2}\right)^\frac{1}{n} \left(\|T^n\|_A\right)^\frac{1}{n}\le \lim_{n \to \infty}\left(w_{q,A}(T^n)\right)^\frac{1}{n} \le \lim_{n \to \infty}\left(\|T^n\|_A\right)^\frac{1}{n}.
 \end{equation*}
As $q \ne 0$, it follows from equation \eqref{srf} that
\begin{equation*}
 \lim_{n \to \infty}\left(w_{q,A}(T^n)\right)^\frac{1}{n}=r_A(T).   
\end{equation*}
\end{proof}

Next, we establish a relation between the $A$-$q$-numerical range of $T$ in semi-Hilbert space $\mathcal{H}$ and $q$-numerical range of $T$ in Hilbert space $\textbf{R}(A^{1/2})$.
The following lemma is essential in this regard.  
\begin{lemma}\label{lemma-tilde}\cite{arias2009lifting}
 Let \( T \in \mathcal{B}(\mathcal{H}) \). Then \( T \in \mathcal{B}_{A^{1/2}}(\mathcal{H}) \) if and only if there exists a unique \( \widetilde{T} \in \mathcal{B}(\textbf{R}(A^{1/2})) \) such that \( Z_A T = \widetilde{T} Z_A \). Here, \( Z_A : \mathcal{H} \to \textbf{R}(A^{1/2}) \) is defined by \( Z_A x = Ax \) for all \( x \in \mathcal{H} \).   
\end{lemma}
\begin{theorem}\label{tilde}
   Let $T \in \mathcal{B}_{A^\frac{1}{2}}(\mathcal{H})$. Then
$W_{q,A}(T) \subseteq W_q(\widetilde{T})$.
\end{theorem}
\begin{proof}
Let $\widetilde{T} \in \mathcal{B}(\textbf{R}(A^\frac{1}{2}))$. Then, by using decomposition $\mathcal{H}= \overline{R(A^{1/2})} \oplus N(A^{1/2})$,  we get
\begin{align*}  
&W_q(\widetilde{T})= \{ (\widetilde{T}x, y) :x, y \in \textbf{R}(A^{1/2}), \|x\|_{\textbf{R}(A^{1/2})} = 1= \|y\|_{\textbf{R}(A^{1/2})}, (x,y)_{\textbf{R}(A^{1/2})}=q \} \\  
&= \{ (\widetilde{T}A^{1/2}x, A^{1/2}y) : x,y \in \mathcal{H}, \|A^{1/2}x\|_{\textbf{R}(A^{1/2})} = 1= \|A^{1/2}y\|_{\textbf{R}(A^{1/2})}, ( A^{1/2}x,A^{1/2}y)_{\textbf{R}(A^{1/2})} = q\} \\  
&= \{ (\widetilde{T}A^{1/2}x, A^{1/2}y) : x ,y\in \overline{R(A^{1/2})}, \|A^{1/2}x\|_{\textbf{R}(A^{1/2})} = 1=\|A^{1/2}y\|_{\textbf{R}(A^{1/2})}, (A^{1/2}x,A^{1/2}y)_{\textbf{R}(A^{1/2})}=q  \}. 
\end{align*}  
Also, by using decomposition $\mathcal{H}= \overline{R(A^{1/2})} \oplus N(A^{1/2})$ and Lemma \ref{lemma-tilde}, we have
\begin{align*}
W_{q,A}(T) &= \{ \langle Tx, y \rangle_A : x,y \in \mathcal{H}, \|x\|_A = 1=\|y\|_A, \langle x,y \rangle_A=q \} \\  
&= \{ (ATx, Ay) : x,y \in \mathcal{H}, \|Ax\|_{\textbf{R}(A^{1/2})} = 1=\|Ay\|_{\textbf{R}(A^{1/2})},(Ax,Ay)_{\textbf{R}(A^{1/2})}=q \} \\  
&= \{ (\widetilde{T}Ax, Ay) : x,y \in \overline{R(A^{1/2})}, \|Ax\|_{\textbf{R}(A^{1/2})} = 1=\|Ay\|_{\textbf{R}(A^{1/2})},(Ax,Ay)_{\textbf{R}(A^{1/2})}=q \}.
\end{align*} 
 Thus, $W_{q,A}(T)\subseteq W_q(\widetilde{T}).$
\end{proof}
Using the fact that $R(A)=R(A^\frac{1}{2})$, when $R(A)$ is closed, we can obtain the forthcoming corollary easily.
\begin{corollary}\label{eqcor}
 Let $R(A)$ is closed. Then $W_{q,A}(T)=W_q(\widetilde{T}).$ 
\end{corollary}


We now provide an example demonstrating that the assumption of the closedness of \( R(A) \) is not necessary for Corollary \ref{eqcor}.
\begin{example}
Let $\{e_n\}_{n=1}^\infty$ denote the standard orthonormal basis of the Hilbert space $\ell_2$. Let $A: \ell_2 \to  \ell_2$ be the bounded linear operator defined as
\begin{equation*}
    A(e_n)=\frac{1}{n^p}e_n,
\end{equation*}
where $p \ge 2$. In this case, $A$ is a positive operator and $R(A)$ is not closed. Consider the bounded linear operator $T: \ell_2 \to  \ell_2$ defined as
\begin{equation*}
    Te_n=c  e_n
\end{equation*}
where $c$ is an arbitrary constant and $|c| \le 1$. By applying Lemma \ref{lemma-tilde}, we obtain $\widetilde{T} \in \mathcal{B}(\textbf{R}(A^{1/2}))$ defined as
\begin{equation*}
 \widetilde{T}e_n= c e_n.   
\end{equation*}
Then it is easy to verify that 
\begin{equation*}
 W_{q,A}(T)=W_q(\widetilde{T})=\{cq\}.   
\end{equation*}
\end{example}

Before proving our next result, first we state the following results.

\begin{lemma}\label{slemma} \cite[Theorem 4.2]{majumdar2024approximate}
Let $T \in \mathcal{B}_{A}(\mathcal{H})$. Then $T$ is $A$-selfadjoint if and only if $\widetilde{T}$ is selfadjoint.
\end{lemma}

\begin{lemma}\label{rlemma} \cite[Proposition 2.5(ii)]{sen2024note}
Let $T \in \mathcal{B}_{A^{1/2}}(\mathcal{H})$ with $R(A)$ is closed. Then $\sigma_{A_P}(T) =\sigma_p(\widetilde{T})$.  
\end{lemma}

Now we are ready to prove the next result which establishes that the $A$-$q$-numerical range of an $n \times n$ $A$-self-adjoint matrix is an elliptic disk.
\begin{theorem}
 Let $|q| \le 1$ and $T$ be an $n \times n$ $A$-self-adjoint matrix. Then $T$ has $\lambda_1 \ge \lambda_2,... \ge \lambda_m$ $A$-eigenvalues, where $\dim(R(A^\frac{1}{2})=m \le n$ and $W_{q,A}(T)$ is the closed elliptic disk with foci $q \lambda_1$ and $q \lambda_m$ and minor axis of length $\sqrt{1-|q|^2}(\lambda_1- \lambda_m).$   
\end{theorem}
\begin{proof}
Since, $T$ is an $n \times n$ $A$-selfadjoint matrix, Lemma \ref{slemma} implies $\widetilde{T}$ is self adjoint. Also, $\widetilde{T} \in \mathcal{B}(\textbf{R}(A^{1/2}))$ and $\dim(R(A^\frac{1}{2})=m$, thus 
 $\lambda_1 \ge \lambda_2,... \ge \lambda_m$ are $m$ eigenvalues of $\widetilde{T}$. Since $R(A)$ is closed, using Lemma \ref{rlemma}, we have
  $\sigma_{A_P}(T) =\sigma_p(\widetilde{T})$. This implies, $\lambda_1 \ge \lambda_2,... \ge \lambda_m$ are $m$ $A$-eigenvalues of $T$.  Moreover, using Theorem 3.4 \cite[p.384]{gau2021numerical}, we have $W_q(\widetilde{T})$ is the closed elliptic disk with foci $q \lambda_1$ and $q \lambda_m$ and minor axis of length $\sqrt{1-|q|^2}(\lambda_1- \lambda_n).$ The required result follows from Corollary \ref{eqcor}.    
\end{proof}

\section{$A$-nilpotent operator}
In this section we introduce the concept of $A$-nilpotent operator and study its $A$-$q$-numerical range and radius. We start with the definition of $A$-nilpotent operator.
\begin{definition}
 An operator $T \in \mathcal{B}_{A^{1/2}}(\mathcal{H})$ is said to be $A$-nilpotent of index $k \in \mathbb{N}$ if $AT^k=0$ but $AT^{k-1} \ne 0$.  
\end{definition}

By definition it is clear that every nilpotent operator is $A$-nilpotent with possible index $1$. In fact, consider the matrices $T=\begin{pmatrix}
    0 & \ell \\
    0 & 0
\end{pmatrix}$ and $A=\begin{pmatrix}
    0 & 0 \\
    0 & m
\end{pmatrix}$, where $\ell, m$ are real numbers. Then it is easy to verify that $T$ is nilpotent operator of index 2 and $A$-nilpotent operator of index $1$. However, the converse is not necessarily true. For example, consider the matrices $A=\begin{pmatrix}
  1 & 0 & 0 \\
  0 & 0 & 0 \\
  0 & 0 & 0
\end{pmatrix} $ and $T=\begin{pmatrix}
  0 & 1 & 0 \\
  0 & 0 & 0 \\
  0 & 0 & 2
\end{pmatrix} $. We obtain, $T^n \ne 0$ for all $n \in \mathbb{N}$, $AT \ne 0$ but $AT^2=0$. In this case $T$ is not a nilpotent operator but $T$ is $A$-nilpotent operator of index $2$.
\begin{theorem}\label{nil}
     $T \in \mathcal{B}_{A^{1/2}}(\mathcal{H})$ is $ A$-nilpotent operator of index $k$ if and only if $\widetilde{T}$ is nilpotent operator of index $k$. 
\end{theorem}
\begin{proof}
 Let $T$ is $A$-nilpotent operator of index $k$. Then we have $AT^kx=0$ for all $x \in \mathcal{H}$. Also, by applying Lemma \ref{lemma-tilde}, we obtain $\widetilde{T}^kAx=AT^kx$. Thus, $\widetilde{T}^kAx=0$ for all $x \in \mathcal{H}$. Since $R(A)$ is dense in $\textbf{R}(A^\frac{1}{2})$, we have $\widetilde{T}^k=0$. Also, if $\widetilde{T}^{k-1}=0$ then $AT^{k-1}x=\widetilde{T}^{k-1}Ax=0$ for all $x \in \mathcal{H}$. It follows that $AT^{k-1}=0$, which is a contradiction. Therefore, $\widetilde{T}^{k-1}\ne 0$ and hence $\widetilde{T}$ is nilpotent operator of index $k.$ Conversely, suppose that $\widetilde{T}$ is nilpotent operator of index $k$. Then from Lemma \ref{lemma-tilde}, we have $AT^kx=\widetilde{T}^kAx=0$ for all $x \in \mathcal{H}$, this implies that $AT^k=0$. Again if $AT^{k-1}=0$, we obtain $\widetilde{T}^{k-1}Ax=AT^{k-1}x=0$ for all $x \in \mathcal{H}$. Therefore, $\widetilde{T}^{k-1}=0$, which is not possible. Thus $AT^{k-1} \ne 0$ and this proves the result.       
\end{proof}

In the following lemma we prove some basic results on $A$-nilpotent operators. 
\begin{lemma}
The following properties holds.
\begin{itemize}
    \item [(i)] $T \in \mathcal{B}_{A}(\mathcal{H})$ is $ A$-nilpotent operator of index $k$ if and only if $T^\#$ is $A$-nilpotent operator of index $k$.

 \item [(ii)] Let $T \in \mathcal{B}_{A^{1/2}}(\mathcal{H})$ is $ A$-nilpotent operator of index $k$. Then $\sigma_A(T)=\{0\}$.  
\end{itemize}    
\end{lemma}
\begin{proof}
Since $T$ is $ A$-nilpotent operator of index $k$, we have $AT^k=0$ but $AT^{k-1} \ne 0$.
\begin{itemize}
\item [(i)]  Now,
\begin{align*}
    A(T^\#)^k= & AT^\#(T^\#)^{k-1}\\
    =& T^*A(T^\#)^{k-1}\quad \text{(as $AT^\#=T^*A$)}\\
    =& (T^*)^k A\\
    =& (AT^k)^*
\end{align*}
Thus, $AT^k=0 \Leftrightarrow A(T^\#)^k=0$. Also, $AT^k \ne 0 \Leftrightarrow A(T^\#)^k \ne 0$. Hence, the required result holds.
\item[(ii)] Using $A$-spectral radius formula of $T$ mentioned in equation \eqref{srf}, we have 
\begin{align*}
    r_A(T)= 
        &  \lim_{ n \to \infty}\|A^\frac{1}{2}T^n\|^\frac{1}{n}
\end{align*}
As $A$ is a positive operator, $AT^k=0$ implies $A^\frac{1}{2}T^k=0$. Thus, $r_A(T)=0$ and hence $\sigma_A(T)=\{0\}$. 
\end{itemize}
\end{proof}

By applying a similar approach to \cite[Theorem 5]{karaev2010numerical} and \cite[Theorem 5.10]{stankovic2024some}, we obtain the following lemma for the $q$-numerical range of square zero operators.
\begin{lemma}\label{lnilpotent}
    	Let $\mathcal{H}$ be a complex Hilbert space, $T \in \mathcal{B}(\mathcal{H})$ and $q \in \mathcal{D}$. If $T$ is a nilpotent operator with index $2$, then the $q$-numerical range of the operator $T$ is a disk (open or closed) with center at $0$ and $w_q(T) \le \left(\frac{1+\sqrt{1-|q|^2}}{2}\right)\|T\|$.
\end{lemma}
\begin{proof}
 As $T^2=0$ and $T \ne 0,$ we have $N(T) \ne \{0\}$. This implies $0 \in \sigma_p(T)$. Since $q \sigma_p(T)\subseteq W_q(T)$\cite{feki2024joint}, we have $0 \in {W_q(T)}$. To prove $W_q(T)$ is a circular set, let $\lambda \in W_q(T)$. We can write $\lambda= \langle Tx,y \rangle$ where $\|x\|=\|y\|=1$ and $\langle x,y \rangle=q$. Then for any $t \in [0, 2\pi]$, we have
	$$e^{it}\lambda=e^{it}\langle Tx,y \rangle=\langle e^{it} Tx,y \rangle.$$ 
	Also, the square zero operator can be written as $\begin{pmatrix}
		0 & S\\
		0 & 0
	\end{pmatrix} \in \mathcal{B}(\mathcal{H}_1 \oplus \mathcal{H}_2)$ for some operator $S$ where $\mathcal{H}=\mathcal{H}_1 \oplus \mathcal{H}_2$. Taking $ U=\begin{pmatrix}
		I & 0\\
		0 & e^{it}I
	\end{pmatrix}$, where $I$ is the identity operator, we have the relation $$TU=e^{it}UT.$$
	This shows that $T$ and $e^{it}T$ are unitarily equivalent operators. Therefore, by Proposition 3.1(d) \cite[p.380]{gau2021numerical}, we have 
	\begin{equation*}
		W_q(e^{it}T)=W_q(T).
	\end{equation*}
	 Thus, $e^{it}\lambda \in W_q(T)$ for each $t \in [0,2\pi]$. This implies $W_q(T)$ is a circular set with origin as center. From the convexity of the $q$-numerical range we have that $W_q(T)$ is a disk (open or closed) with center at $0$. 
     
     To obtain the upper bound of $\omega_{q}(T),$ let $x=x_1 \oplus x_2, ~y=y_1 \oplus y_2 \in \mathcal{B}(\mathcal{H}_1 \oplus \mathcal{H}_2)$, where \[\|x\|^2=\|x_1\|^2+\|x_2\|^2=1, \ \|y\|^2=\|y_1\|^2+\|y_2\|^2=1, \ \mbox{and}\] 
     \[\langle x,y \rangle =\langle x_1, y_1 \rangle+\langle x_2, y_2 \rangle=q.\] 
     We can consider $\|x_1\|=\cos \alpha$, $\|x_2\|=\sin \alpha$, $\|y_1\|=\cos \beta$ and $\|y_2\|=\sin \beta$ for some $\alpha, \beta \in [0, \frac{\pi}{2}]$. Then
\begin{equation*}
    \langle Tx,y \rangle = \langle \begin{pmatrix}
    0 & S\\
    0 & 0
\end{pmatrix}\begin{pmatrix}
    x_1\\
    x_2
\end{pmatrix}, \begin{pmatrix}
    y_1\\
    y_2
\end{pmatrix}\rangle=\langle S x_2, y_1 \rangle.
\end{equation*}
To find $w_q(T)$, we have to find supremum value of $| \langle Sx_2, y_1 \rangle|$. Using Cauchy-Schwarz inequality, we have 
\begin{align*}
 | \langle Sx_2, y_1 \rangle| \le & \|S\|\|x_2\| \|y_1\|\\
 = &\|S\| \sin \alpha \cos \beta\\
 = &\|S\| \left( \frac{1}{2}\sin (\alpha-\beta)+\frac{1}{2}\sin(\alpha+\beta) \right).   
\end{align*}
Also, 
\begin{equation*}
  |q| =|\langle x_1, y_1 \rangle + \langle x_2, y_2 \rangle| \le  \|x_1\|\|y_1\|+\|x_2\|\|y_2\|=\cos(\alpha-\beta). 
\end{equation*}
This implies, $\sin(\alpha-\beta) \le \sqrt{1-|q|^2}$. We obtain
\begin{align*}
 \sup | \langle S x_2, y_1 \rangle | \le  &\left( \frac{1+\sqrt{1-|q|^2}}{2}\right)\|S\|\\
 = & \left( \frac{1+\sqrt{1-|q|^2}}{2}\right) \|T\|.   
\end{align*}
Therefore, we have
\begin{equation*}
    w_q(T) \le  \left( \frac{1+\sqrt{1-|q|^2}}{2}\right)\|T\|.
\end{equation*}
\end{proof}

\begin{remark}
Here we prove the refinement of the above result in comparison to the existing upper mentioned in \cite[Theorem 2.5]{fakhri2024q} as follows
\begin{equation*}
 w_q(T) \le  \left( 1-\frac{3q^2}{4}+q \sqrt{1-q^2}\right)^\frac{1}{2}\|T\|~\text{where}~q\in[0,1).    
\end{equation*}
For $q \in[0,1)$, from Lemma \ref{lnilpotent} it follows
\begin{equation*}
    w_q(T) \le  \left( \frac{1+\sqrt{1-|q|^2}}{2}\right)\|T\|.
\end{equation*}
The refinement follows from the fact that
\begin{align*}
 \left( \frac{1+\sqrt{1-|q|^2}}{2}\right)^2- \left( 1-\frac{3q^2}{4}+q \sqrt{1-q^2}\right)
 =\frac{q^2-1+\sqrt{1-q^2}(1-2q)}{2}\le 0.
\end{align*}
\end{remark}

 The following theorem describes the structure of the $A$-$q$-numerical range of a $A$-nilpotent operator with index $2$.
\begin{theorem}\label{nilpotent}
    	 If $T \in \mathcal{B}_{A^{1/2}}(\mathcal{H})$ is a $A$-nilpotent operator with index $2$ and $R(A)$ is closed, then the $A$-$q$-numerical range of $T$ is a disk (open or closed) with center at $0$ and $w_{q,A}(T) \le \left(\frac{1+\sqrt{1-|q|^2}}{2}\right)\|T\|_A$.
\end{theorem}
\begin{proof}
    As $T$ is $A$-nilpotent operator of index $2$ then it follows from Lemma \ref{nil} that $\widetilde{T}$ is nilpotent operator of index $2$. Thus, by Lemma \ref{lnilpotent}, we have $W_q(\widetilde{T})$ is a disk with center at $0$. Theorem \ref{tilde}(ii) shows that $W_{q,A}(T)$ is a disk with center at $0$. Again, using Lemma \ref{lnilpotent}, we have 
    $
    w_q(\widetilde{T}) \le  \left( \frac{1+\sqrt{1-|q|^2}}{2}\right)\|\widetilde{T}\|_{\textbf{R}(A^\frac{1}{2})}.
$
By applying Theorem \ref{tilde}, we have $ w_q(\widetilde{T})=w_{q,A}(T)$. Using this relation and $\|T\|_A=\| \widetilde{T}\|_{\textbf{R}(A^\frac{1}{2})}$ \cite{feki2020spectral}, we obtain
\begin{equation*}
    w_{q,A}(T) \le  \left( \frac{1+\sqrt{1-|q|^2}}{2}\right)\|T\|_A.
\end{equation*}
\end{proof}
\begin{corollary}
 Let $T \in \mathcal{B}_A(\mathcal{H})$, $AT^2=0$ but $AT \ne 0$. Then $w_A(T)= \frac{\|T\|_A}{2}$.    
\end{corollary}
\begin{proof}
By setting $q = 1$ in Theorem~\ref{nilpotent}, we obtain 
\[
    w_A(T) \leq \frac{\|T\|_A}{2}.
\]
On the other hand, taking $q = 1$ in Corollary~\ref{maincor}(ii) yields  
\[
    \frac{\|T\|_A}{2} \leq w_A(T).
\]
Combining the above two inequalities, we conclude that  
\[
    w_A(T) = \frac{\|T\|_A}{2},
\]
and the proof is complete. 
\end{proof}

We have successfully obtained a class of operators for which equality holds in the estimnation of $w_q(T)$ obtained in Lemma \ref{lnilpotent}. This proves that the estimation of $w_q(T)$ obtained in Lemma \ref{lnilpotent} is best possible.
\begin{corollary}\label{sqcor}
 Let $q \in [0,1]$ and $T=\begin{pmatrix}
    0 & S\\
    0 & 0
\end{pmatrix} \in \mathcal{B}(\mathcal{H} \oplus \mathcal{H})$, where $S$ is compact, self-adjoint operator. Then
$
    w_q(T) =  \left( \frac{1+\sqrt{1-q^2}}{2}\right)\|T\|.
$
\end{corollary}
\begin{proof}
 From Lemma \ref{nilpotent}, we already have $w_q(T) \le  \left( \frac{1+\sqrt{1-q^2}}{2}\right)\|T\|.$ Since $S$ is compact, self-adjoint operator, then $\|S\|$ or $-\|S\|$ is an eigenvalue of $S$ and let $e$, $(\|e\|=1)$ be the corresponding eigenvector. 
 Let  $x=\begin{pmatrix}
    x_1\\
    x_2
\end{pmatrix}=\begin{pmatrix}
    (\cos\alpha)e\\
    (\sin \alpha) e
\end{pmatrix}$ and $y=\begin{pmatrix}
    y_1\\
    y_2
\end{pmatrix} =\begin{pmatrix}
    (\cos \beta) e\\
    (\sin \beta) e
\end{pmatrix}\in \mathcal{H} \oplus \mathcal{H}$, 
  where $\alpha -\beta =\cos^{-1}q$ and $\alpha +\beta =\frac{\pi}{2}$. This implies, $\sin(\alpha -\beta)=\sqrt{1-q^2}$. Clearly, $\|x\|=\|y\|=1$ and $\langle x,y \rangle=q$. Now,
 \begin{align*}
|\langle Tx, y \rangle |=  & | \langle S x_2, y_1 \rangle |\\
=&\sin \alpha \cos \beta | \langle S e, e \rangle |\\
=&\frac{1}{2}\left( \sin (\alpha+\beta) +\sin (\alpha -\beta) \right) \langle \|S\| e, e \rangle\\
=&\left(\frac{1+\sqrt{1-q^2}}{2}\right)\|S\|\\
=&\left(\frac{1+\sqrt{1-q^2}}{2}\right)\|T\|.
 \end{align*}
This implies, 
 $
     w_q(T) \ge \left(\frac{1+\sqrt{1-q^2}}{2}\right)\|T\|.
$ Thus, the required result holds.
\end{proof}
\begin{remark}
    Take $q=1$ in Corollary \ref{sqcor}, we retrieve the result $w(T)=\frac{\|T\|}{2}$ for square zero operators \cite[Corollary 1]{kittaneh2003numerical}.
\end{remark}

We conclude this section by providing an estimation of the $q$-numerical radius of a class of nilpotent operators. 
\begin{theorem}\label{lthm}
Let $\mathcal{H}=\mathcal{H}_1 \oplus \mathcal{H}_2 \oplus \mathcal{H}_3$, $q \in \mathcal{D}$, and $T=\begin{pmatrix}
    0 & S_1& 0\\
    0 & 0 &S_2\\
    0 & 0 & 0
\end{pmatrix} \in \mathcal{B}(\mathcal{H})$. Then $q$-numerical range of $T$ satisfies the following relation 
\begin{equation*}
    w_q(T) \le 
    \begin{cases} 
        \sqrt{2} \max \{ \|S_1\|, \|S_2\| \}, & \text{if $|q| \in \left[0, \frac{1}{\sqrt{2}}\right] $}, \\
     \left( \frac{\sqrt{2} + |q| + \sqrt{1 - |q|^2}}{2} \right)\max \{ \|S_1\|, \|S_2\| \} , & \text{if $|q| \in \left( \frac{1}{\sqrt{2}}, 1 \right]$}.
    \end{cases}
\end{equation*}
\end{theorem}
\begin{proof}
Consider $x=\begin{pmatrix}
    x_1\\
    x_2\\
    x_3
\end{pmatrix}$ and $y=\begin{pmatrix}
    y_1\\
    y_2\\
    y_3
\end{pmatrix}$ in $\mathcal{H}$ such that $\|x\|=\|y\|=1$ and $\langle x,y \rangle=q$. Using the spherical coordinate system, we can take $\|x_1\|=\cos\theta_1 \sin \phi_1, ~ \|x_2\|=\sin\theta_1 \sin\phi_1, ~ \|x_3\|=\cos \phi_1$ and $\|y_1\|=\cos\theta_2 \sin \phi_2, ~ \|y_2\|=\sin\theta_2 \sin\phi_2, ~ \|y_3\|=\cos \phi_2$ where $\phi_1, \phi_2 \in [0, \frac{\pi}{2}]
$ and $\theta_1, \theta_2 \in [0, \frac{\pi}{2}]$. Using Cauchy-Schwarz inequality, we have 
\begin{align*}
 |q|=&|\langle x,y \rangle| \le\|x_1\|\|y_1\|+\|x_2\|\|y_2\|+\|x_3\|\|y_3\|\\
 =& \cos\theta_1 \sin \phi_1 \cos\theta_2 \sin \phi_2+ \sin\theta_1 \sin\phi_1 \sin\theta_2 \sin\phi_2 + \cos \phi_1 \cos\phi_2 \\
 =& \sin \phi_1 \sin \phi_2 \cos(\theta_1-\theta_2)+ \cos\phi_1 \cos\phi_2\\
 \le & \sin \phi_1 \sin \phi_2+ \cos\phi_1 \cos\phi_2\\
 =& \cos(\phi_1-\phi_2).
\end{align*}
This implies, $\sin(\phi_1-\phi_2) \le \sqrt{1-|q|^2}$.
We have
\begin{align*}
& \left\langle \begin{pmatrix}
     0 & S_1 & 0 \\
     0 & 0 & S_2 \\
     0 & 0 & 0
 \end{pmatrix} \begin{pmatrix}
     x_1 \\
     x_2\\
     x_3
 \end{pmatrix}, \begin{pmatrix}
     y_1 \\
     y_2\\
     y_3
 \end{pmatrix}\right\rangle\\
 =&\langle S_1 x_2,y_1 \rangle + \langle S_2 x_3, y_2 \rangle\\ 
 \le & \max \{ \|S_1\|, \|S_2\| \}(\|x_2\|\|y_1\|+\|x_3\|\|y_2\|)  ~~ \quad \text{ (Cauchy-Schwarz inequality)}\\
\le & \max \{ \|S_1\|, \|S_2\| \}(\sin\theta_1 \sin\phi_1\cos\theta_2 \sin \phi_2+\cos \phi_1\sin\theta_2 \sin\phi_2) \\
\le & \max \{ \|S_1\|, \|S_2\| \} \left( \sin \phi_1 \sin \phi_2 +\cos \phi_1 \sin \phi_2 \right)\\
 =&\frac{\max \{ \|S_1\|, \|S_2\| \}}{2}\left(\cos(\phi_1-\phi_2)-\cos(\phi_1+\phi_2)+\sin(\phi_1+\phi_2)-\sin(\phi_1-\phi_2)\right)\\
=&\frac{\max \{ \|S_1\|, \|S_2\| \}}{2}\left(\sin(\phi_1+\phi_2) -\cos(\phi_1+\phi_2)+\cos(\phi_1-\phi_2)-\sin(\phi_1-\phi_2)\right)\\
\le&\frac{\max \{ \|S_1\|, \|S_2\| \}}{2}\left( \sqrt{2} +\cos(\phi_1-\phi_2)-\sin(\phi_1-\phi_2)\right)
\end{align*}
Let $\psi=\phi_1-\phi_2$. Now the problem reduces to finding the supremum of $\cos\psi-\sin\psi$ with $|q| \le \cos \psi$ and $\psi \in [\frac{-\pi}{2}, \frac{\pi}{2}].$ Thus, $\psi \in \left[ -\cos^{-1}(|q|), \cos^{-1}(|q|)\right]=S (\text{say})$.
It is evident that $\cos \psi -\sin\psi$ attains its supremum value at $\psi=-\frac{\pi}{4}$, which is $\sqrt{2}$. Next we consider two cases:

\textit{Case I: $|q| \in \left[0, \frac{1}{\sqrt{2}}\right]$.}

In this case $-\frac{\pi}{4}\in S$. Thus, $\sup_\psi \{ \cos(\psi)-\sin(\psi)\}=\sqrt{2}$. Therefore,
\begin{equation*}
   \left\langle \begin{pmatrix}
     0 & S_1 & 0 \\
     0 & 0 & S_2 \\
     0 & 0 & 0
 \end{pmatrix} \begin{pmatrix}
     x_1 \\
     x_2\\
     x_3
 \end{pmatrix}, \begin{pmatrix}
     y_1 \\
     y_2\\
     y_3
 \end{pmatrix}\right\rangle \le \sqrt{2}\max \{ \|S_1\|, \|S_2\| \}    
\end{equation*}
and 
\begin{equation*}
w_q(T) \le \sqrt{2}\max \{ \|S_1\|, \|S_2\| \}.   
\end{equation*}
\textit{Case II: $|q| \in \left(\frac{1}{\sqrt{2}},1\right]$.}

In this case $-\frac{\pi}{4}\notin S$. Clearly, $\cos\psi-\sin\psi$ is a decreasing function, so it attains a supremum at the initial point. Therefore,
\begin{align*}
  \sup_\psi \{ \cos\psi-\sin\psi\}
  =\cos (-\cos^{-1}(|q|)-\sin (-\cos^{-1}(|q|)
  =|q|+\sqrt{1-|q|^2}.
\end{align*}
Thus, 
\begin{equation*}
  \sup_\psi \{|\langle S_1 x_2,y_1 \rangle + \langle S_2 x_3, y_2 \rangle|\} 
 \le \left(\frac{\sqrt{2}+|q|+\sqrt{1-|q|^2}}{2} \right)\max \{ \|S_1\|, \|S_2\| \} ~\text{and}
 \end{equation*}
 \begin{equation*}
  w_q(T) \le \left(\frac{\sqrt{2}+|q|+\sqrt{1-|q|^2}}{2} \right)\max \{ \|S_1\|, \|S_2\| \}. 
\end{equation*}
Hence, the required result holds.
\end{proof}
Since the bounds for the classical numerical radius of a nilpotent operator of index three have not been investigated in the existing literature, in the following corollary, we obtain such bounds.
\begin{corollary}
 Let $q=1$ and $T=\begin{pmatrix}
    0 & S_1& 0\\
    0 & 0 &S_2\\
    0 & 0 & 0
\end{pmatrix} \in \mathcal{B}(\mathcal{H})$, it follows from Theorem \ref{lthm} that
\begin{equation*}
    w(T) \le \frac{1+\sqrt{2}}{2}\max \{ \|S_1\|, \|S_2\| \}.
\end{equation*}
\end{corollary}
\begin{remark}
    If we take \( q \in \mathcal{D} \) and 
$
    T = \begin{pmatrix}
        0 & S_1 & 0 \\
        0 & 0 & S_2 \\
        0 & 0 & 0
    \end{pmatrix},
$
    where \( S_1, S_2 \in \mathcal{B}_{A^{1/2}}(\mathcal{H}) \) and 
$
    \mathbb{A} = \begin{pmatrix}
        A & 0 & 0 \\
        0 & A & 0 \\
        0 & 0 & A
    \end{pmatrix} \in \mathcal{B}(\mathcal{H} \oplus \mathcal{H} \oplus \mathcal{H})^{+},
$
    then Theorem \ref{lthm} also holds in the setting of semi-Hilbertian space.
\end{remark}

\end{document}